\documentclass[11pt]{article}  
\usepackage{amsmath}
\usepackage{amssymb}
\usepackage{theorem}
\usepackage{euscript}
\usepackage{pstricks}
\usepackage{authblk}
\usepackage{verbatim}
\usepackage{hyperref}

\usepackage{ifthen}
\usepackage{xifthen}
\usepackage{xcolor}
\usepackage{fullpage}
\usepackage[english]{babel}
\usepackage{mathtools}
\usepackage{hyperref}
\usepackage{color}

\usepackage{hyperref}
\hypersetup
{
	colorlinks,
	citecolor=black,
	filecolor=black,
	linkcolor=black,
	urlcolor=magenta
}
\hypersetup{linktocpage} 

\newtheorem{theorem}{Theorem}[section]
\newtheorem{lemma}[theorem]{Lemma}

\numberwithin{equation}{section}

\newtheorem{remark}{Remark}[section]
\numberwithin{equation}{section}

\def\ZZ{{\mathbb Z}}

\def\RRd{{\mathbb R}^d}

\def\CC{{\mathbb C}}
\def\CC{{\mathbb C}}

\def\NN{{\mathbb N}}
\def\RR{{\mathbb R}}

\def\UU{{\mathbb U}}

\def\RRd{{\mathbb R}^d}

\def\RRi{{\mathbb R}^\infty}

\def\RRi{{\mathbb R}^\infty}

\def\RRp{{\mathbb R}_+}
\def\RRip{{\mathbb R}^{\infty}_+}

\def\UUi{{\mathbb U}^{\infty}}

\def\Dd{{\mathcal D}}

\def\Jj{{\mathcal J}}

\def\CC{{\mathbb C}}
\def\ZZ{{\mathbb Z}}
\def\NN{{\mathbb N}}
\def\RR{{\mathbb R}}

\def\FF{{\mathbb F}}

\def\RRd{{\mathbb R}^d}

\def\supp{\operatorname{supp}}
\def\dv{\operatorname{div}}

\newcommand{\bb}{{\boldsymbol{b}}}

\newcommand{\be}{{\boldsymbol{e}}}

\newcommand{\bk}{{\boldsymbol{k}}}

\newcommand{\bp}{{\boldsymbol{p}}}

\newcommand{\bs}{{\boldsymbol{s}}}

\newcommand{\bx}{{\boldsymbol{x}}}

\newcommand{\by}{{\boldsymbol{y}}}

\newcommand{\bone}{{\boldsymbol{1}}}

\newcommand{\brho}{{\boldsymbol{\rho}}}
\newcommand{\bsigma}{{\boldsymbol{\sigma}}}

\newcommand{\balpha}{{\boldsymbol{\alpha}}}

\newcommand{\bbeta}{{\boldsymbol{\beta}}}

\newcommand{\bgamma}{{\boldsymbol{\gamma}}}
\newcommand{\blambda}{{\boldsymbol{\lambda}}}

\newcommand{\rd}{{\rm d}} 
 
\newcommand{\norm}[2]{\left\|{#1}\right\|_{#2}}

\newcommand{\brab}[1]{\left\{#1\right\}}
\newcommand{\brac}[1]{\left(#1\right)}

\title{\sffamily {Sparsity  for parametric PDEs with log-gamma random inputs and applications}}

\author[a]{Dinh D\~ung}
\affil[a]{Information Technology Institute, Vietnam National University, Hanoi
	\protect\\
	144 Xuan Thuy, Cau Giay, Hanoi, Vietnam
	\protect\\
	Email: dinhzung@gmail.com}
	
	\author[b]{Viet Ha Hoang}
\affil[b]{Division of Mathematical Sciences, School of Physical and Mathematical Sciences, Nanyang
	Technological University
		\protect\\ 637371 Singapore
		\protect\\ vhhoang@ntu.edu.sg}
		
\medskip
\author[c]{Van Kien Nguyen}
\affil[c]{Department of Mathematical Analysis, University of Transport and Communications
	\protect\\	No.3 Cau Giay Street, 	Hanoi, Vietnam
	\protect\\
	Email: kiennv@utc.edu.vn}

\date{\today}
\date{\ttfamily  \today}
\begin{document}
\maketitle

\begin{abstract}	
	We propose a  novel method for establishing the sparsity of  the coefficients of the  Laguerre generalized polynomial chaos expansion
	of  solutions to parametric elliptic PDEs  with  log-gamma inputs on $\mathbb{R}_+^\infty$. 
	 The established sparsity is quantified by $\ell_p$-summability and weighted $\ell_2$-summability of the coefficients.	
	 Building on these sparsity results, we derive convergence rates for semi-discrete approximations in the parametric variables. These rates apply to  sparse-grid polynomial interpolations,  extended least-squares approximations and the associated semi-discrete quadrature rules.
		Moreover, a counterpart of our method for  parametric elliptic PDEs with  log-normal inputs yields a significant improvement in the sufficient condition for 
	$\ell_p$-summability when the component functions  in the log-normal representation of the parametric diffusion coefficients have global support, compared with results obtained in prior works.
	
	
\end{abstract}

\section{Introduction}
\label{Introduction}
In Computational Uncertainty Quantification, efficient approximation of infinite-dimensional parametric PDEs with random inputs has seen significant progress recently. The field is vast, making a full catalog impractical. For a thorough survey and bibliography, see \cite{Adcock2022,CoDe15a,DNSZ2023,GWZ2014,SG2011}.
The primary driver of convergence rates for numerical integration and interpolation of a parametric solution  (as a mapping from the parametric domain to a Bochner space) is the sparsity of its generalized polynomial chaos (GPC) expansion coefficients. This sparsity is quantified by $\ell_p$-summability or weighted $\ell_2$-summability. The authors of  \cite{HoSc14} used real-variable bootstrapping arguments for investigation the sparsity  of  the coefficients of the Hermite GPC  expansion
of  solutions to parametric PDEs  with log-normal inputs based on the Hermite's  second-order differential equation. Real-variable bootstrapping arguments were also used in 
\cite{BCDC17,BCDM17,BCM17} to study the sparsity   of  the coefficients of the  GPC  expansion
of  solutions to parametric elliptic PDEs  with log-normal and affine inputs.  The papers \cite{BCDC17,BCDM17,BCM17} significantly improved the results on sparsity  in the case when the supports of  the component functions in the presentation of the parametric diffusion coefficient have a finite number of overlaps. However, the real-variable bootstrapping arguments suffer certain technical difficulties in generalizing and extending to some important classes of parametric PDEs (cf. \cite{BCDC17,DNSZ2023}). 
An efficient and widely applicable approach to establish the sparsity is the complex  method which is based on analytic continuation  and complex-variable arguments to bound parametric derivatives and GPC expansion coefficients.   This approach has been developed
 in \cite{ADMHolo,CCS13,CDS10,Zec18T,ZDS19} for ``the compact case", where the random parameters of the random  inputs range in compact subsets of $\RR$.  
Unlike in these references, the work \cite{DNSZ2023} investigated this approach for parametric PDEs  with Gaussian field inputs when  the parameter domain $\RR^\infty$   is not compact and the associated probability measure is the standard tensor-product Gaussian measure, and more general for analytic Hilbert-valued functions on $\RRi$.
The results on sparsity and relevant interpolation and quadrature approximation algorithms in \cite{DNSZ2023}
can be applied  to a wide range of parametric and stochastic PDEs with Gaussian field inputs. 

In the present paper, we develop a new approach in  to study the sparsity  for parametric PDEs  with log-gamma random inputs when  the non-compact parameter domain is $\RRip$. The gamma probability distribution has important applications in various fields, including econometrics, Bayesian statistics, biology and medicine (genomics, phylogenomics, neuroscience, oncology...). 
 Unfortunately, the techniques and arguments of  methods employed in 
\cite{ADMHolo,CCS13,CDS10,DNSZ2023,Zec18T,ZDS19}  which are  based particularly on the Rodrigues formula  for the associated orthonormal polynomials, are not suitable for the log-gamma context. 
 We propose a distinct  method for establishing the sparsity of  the coefficients of the  Laguerre GPC expansion
of  solutions to parametric PDEs  with  log-gamma inputs on $\RRip$. This entails significant modifications of mathematical arguments as
compared to those in \cite{CCS13,CDS10,Zec18T,ZDS19,DNSZ2023}. 

   Consider an important model PDE, the divergence-form diffusion elliptic equation 
\begin{equation} \label{ellip}
	- \dv (a\nabla u)
	\ = \
	f \quad \text{in} \quad D,
	\quad u|_{\partial D} \ = \ 0, 
\end{equation}
for   a bounded  Lipschitz domain $D \subset \RRd$, a right-hand side $f$ and a 
spatially variable scalar diffusion coefficient $a$.
Denote by $V:= H^1_0(D)$ the energy space and $V' = H^{-1}(D)$ the dual space of $V$. Assume that  $f \in V'$ and  $a \in L_\infty(D)$ (in what follows this preliminary assumption always holds without mention). If $a$ satisfies the ellipticity assumption
	\begin{equation} \nonumber
		0<\underset{\bx\in D}{\operatorname{ ess inf}}\,a(\bx) \leq a \leq  \underset{\bx\in D}{\operatorname{ ess sup}}\, a(\bx)<\infty,
	\end{equation}
by the well-known Lax-Milgram lemma, there exists a unique weak 
solution $u \in V$  to the equation~\eqref{ellip}  satisfying the variational equation
\begin{equation} \nonumber
	\int_{D} a(\bx)\nabla u(\bx) \cdot \nabla v(\bx) \, \rd \bx
	\ = \
	\langle f , v \rangle,  \quad \forall v \in V.
\end{equation}
For the equation~\eqref{ellip}, 
we consider  the diffusion coefficients having a parametric form $a=a(\by)$, where $\by=(y_j)_{j \in \NN}$
is a sequence of real-valued parameters ranging in $\RRip$ or $\RRi$.
Denote by $u(\by)$ the weak solution to the 
parametric  diffusion divergence-form elliptic equation 
\begin{equation} \label{SPDE}
	- {\rm div} (a(\by)\nabla u(\by))
	\ = \
	f \quad \text{in} \quad D,
	\quad u(\by)|_{\partial D} \ = \ 0,
\end{equation}	
where the parametric diffusion coefficient $a$  is of the form
\begin{equation} \label{lognormal}
	a(\by)=\exp(b(\by)), \quad {\text{with }}\ b(\by)=\sum_{j = 1}^\infty y_j\psi_j, \ \ \by \in \UUi,
	\ \ \text{and} \ \ \psi_j \in L_\infty(D),
\end{equation}
 and 
\begin{itemize}
	\item[{($\operatorname{\Gamma}$)}]
 $y_j$ are   i.i.d.  gamma random variables on $\UU=\RRp$, or
\item[{\rm (G)}]
 $y_j$ are   i.i.d.  standard Gaussian random variables on $\UU=\RR$.
\end{itemize}
The problem of analyticity and sparsity has been studied in \cite[Section 3.6]{DNSZ2023} for the parametric equation with with log-normal random inputs \eqref{SPDE}--\eqref{lognormal}(G).
 To the knowledge of the authors, so far this problem for the  parametric equation with log-gamma random inputs  \eqref{SPDE}--\eqref{lognormal}($\operatorname{\Gamma}$) has not been considered in any prior works. In the present paper, we focus our attention to this problem for equation \eqref{SPDE}--\eqref{lognormal}($\operatorname{\Gamma}$) and briefly revisit it for  equation \eqref{SPDE}--\eqref{lognormal}(G). 

We shortly describe the main contribution of this paper. 
\begin{itemize}
	\item[{\rm (i)}] We propose a  novel method for establishing the sparsity of  the coefficients of the  Laguerre GPC expansion
	of  solutions $u(\by)$ to the parametric equation with log-gamma random inputs  \eqref{SPDE}--\eqref{lognormal}($\operatorname{\Gamma}$). Unlike  approaches in  previous works which  rely particularly on the Rodrigues formula, our method is based specially  on the   Laguerre's second order differential equation for the Laguerre polynomials (for detailed comments, see Remark~\ref{remark2.1}).
	
	\item[{\rm (ii)}] The established sparsity is quantified by $\ell_p$-summability and weighted $\ell_2$-summability of these coefficients.
	 From the sparsity results, as applications, we establish convergence rates for the semi-discrete parametric-variable approximations of $u(\by)$ by sparse-grid polynomial interpolations and extended least squares approximations, as well as for the associated semi-discrete quadratures.
	
	\item[{\rm (iii)}] 
	A counterpart of our method for the parametric equation with log-normal random inputs   \eqref{SPDE}--\eqref{lognormal}(G)  yields a significant improvement in the sufficient condition for 
	$\ell_p$-summability when the component functions  in the log-normal representation of the parametric diffusion coefficients have global support, compared with results obtained in the earlier works \cite{BCDM17,DNSZ2023,HoSc14} (for detailed comments, see Remark~\ref{remark4.1}).	
\end{itemize}

The remaining part of the present paper is organized as follows. 
In Section~\ref{Analyticity and sparsity  for parametric PDEs}, we  propose a method for establishing the sparsity of  the coefficients of the  Laguerre GPC  expansion of  solutions $u(\by)$ to parametric equation with log-gamma random inputs \eqref{SPDE}--\eqref{lognormal}($\operatorname{\Gamma}$). 
In Section~\ref{Semi-discrete parametric approximations}, we prove convergence rates of 
sparse-grid polynomial interpolations,  extended least squares approximation as well as the generated semi-discrete quadratures for $u(\by)$. 
In Section~\ref{Parametric PDEs with Gaussian log-normal inputs}, we extend the sparsity results in Section 
\ref{Analyticity and sparsity  for parametric PDEs}  to the parametric PDEs with log-normal random inputs
\eqref{SPDE}--\eqref{lognormal}(G).

\medskip
\noindent
{\bf Notation} \  As usual, $\NN$ denotes the natural numbers, $\ZZ$  the integers, $\RR$ the real numbers, $\RRp$ the real non-negative numbers, $\CC$ the complex numbers,
$ \NN_0:= \{s \in \ZZ: s \ge 0 \}$.
Denote by $\RR^\infty$ and $\RRip$ the
sets of all sequences $\by = (y_j)_{j\in \NN}$ with $y_j\in \RR$ and $y_j\in \RRp$, respectively.
Denote by $|\by|_0$  the number of nonzero components $y_j$ of  $\by = (y_j)_{j\in \NN}$.
Denote by $\FF$  the set of all sequences of non-negative integers $\bs=(s_j)_{j \in \NN}$ such that their support $\supp (\bs):= \{j \in \NN: s_j >0\}$ is a finite set.  For 
$\bs, \bk \in \FF$ and $k \in \RR$, 
$|\bs|_1:= \sum_{j \in \NN} s_j$, $|\bs|_\infty:=\max\{s_j, j\in \NN\}$, \ $\bs^k := \prod_{j \in \supp(\bs)}s_j^k$, 
$\bs^\bk := \prod_{j \in \supp(\bs)}s_j^{k_j}$, $\bs!:= \prod_{j \in \NN} s_j!$.
If  $\balpha= (\alpha_j)_{j \in \Jj}$  is a set of positive numbers with  any index set $\Jj$, then we use the notation 
$\balpha^{-1}:= (a_j^{-1})_{j \in \Jj}$. Denote $\balpha^\bbeta:= \big(\alpha_j^{\beta_j}\big)_{j \in \Jj}$ for the sets $\balpha= (a_j)_{j \in \Jj}$ and $\bbeta= (\beta_j)_{j \in \Jj}$.
We use letters $C$  and $K$ to denote general 
positive constants which may take different values, and $C_{\alpha,\beta...}$  and $K_{\alpha,\beta,...}$ constants depending on $\alpha,\beta,...$.
Denote by $|G|$ the cardinality of the set $G$.

\section{Sparsity  for parametric PDEs}
\label{Analyticity and sparsity  for parametric PDEs}



\subsection{The Laguerre GPC expansion}
\label{The Laguerre GPC expansion}
For a fixed number $a > 0$, let the gamma probability measure  (associated with gamma  distribution) $\lambda_a$ on $\RRp$ be defined via the density function
\begin{equation} \label{l_a}
	l_a(y):=\frac{1}{\Gamma(a)} e^{-y}y^{a-1},
\end{equation}
where $\Gamma$ is the gamma function.
Let  $\Big(L_k^{(a-1)}\Big)_{k\in \NN_0}$ be the sequence of Laguerre polynomials on $\RRp$ 
normalized with respect to $\lambda_a$, i.e., 
$
\int_{\RRp} |L_k^{(a-1)}(y)|^2  \rd \lambda_a(y)  =1, \ 
k\in \NN_0.
$
For simplicity, we adopt the following abbreviations:
$L_k:=L_k^{(a-1)}$ and  \ $\lambda:=\lambda_a$.

We next recall a concept of
probability measure $\blambda(\by)$ on $\RRip$ as 
the infinite tensor product of the 
measures $\lambda(y_j)$:
\begin{equation*} \label{lambda(by):=}
	\blambda(\by) 
	:= \ 
	\bigotimes_{j \in \NN} \lambda(y_j) , \quad \by = (y_j)_{j \in \NN} \in \RRip.
\end{equation*}
Let $X$ be a separable Hilbert space. For $0<p<\infty$, denote by $L_p(\RRip,X;\blambda)$ the Bochner space of strongly $\blambda$-measurable mappings $v$ from $\RRip$ to $X$, equipped with the norm
\begin{equation*} \label{|v|_{L_2(RRip,X;lambda)}:=}
	\|v\|_{L_p(\RRip,X;\blambda)}
	:= \
	\left(\int_{\RRip} \|v(\by)\|_X^p \, \rd \blambda(\by) \right)^{1/p}.
\end{equation*}
A function $v \in L_2(\RRip,X;\blambda)$ can be represented by the Laguerre GPC expansion
\begin{equation} \label{GPCexpansion}
	v(\by)=\sum_{\bs\in\FF} v_\bs \,L_\bs(\by), \quad v_\bs \in X,
\end{equation}
with convergence in $L_2(\RRip,X;\blambda)$, where
\begin{equation*}
	L_\bs(\by)=\bigotimes_{j \in \NN}L_{s_j}(y_j),\quad 
	v_\bs:=\int_{\RRip} v(\by)\,L_\bs(\by)\, \rd\blambda (\by), \quad 
	\bs \in \FF.
\end{equation*}
Here $\FF$ is the set of all sequences of non-negative integers $\bs=(s_j)_{j \in \NN}$ such that their support 
$\supp (\bs):= \{j \in \NN: s_j >0\}$ is a finite set.
Notice that $(L_\bs)_{\bs \in \FF}$ is an orthonormal basis of $L_2(\RRip,\CC;\blambda)$. 
Moreover, a strongly $\blambda$-measurable function $v$ on $\RRip$ belongs to   $L_2(\RRip,X;\blambda)$ if and only if $v$ is represented by the series \eqref{GPCexpansion} converging in $L_2(\RRip,X;\blambda)$ and  it holds  the Parseval's identity
\begin{equation} \label{Parseval's identity}
	\|v\|_{L_2(\RRip,X;\blambda)}^2
	\ = \ \sum_{\bs\in\FF} \|v_\bs\|_X^2.
\end{equation}

The the Bochner space  $L_p(\RRi,X;\bgamma)$ and the associated infinite tensor product standard Gaussian measure $\bgamma(\by)$ on $\RRi$ can be defined in an analogous   fashion.

\subsection{Sparsity for the Laguerre GPC expansion}
\label{Sparsity for the Laguerre GPC expansion}

We recall  a result on regularity with respect to parametric variables proven in \cite[Lemma 3.9]{DNSZ2023}, which gives bounds for  partial derivatives   of the weak solution $u(\by)$ to equation  \eqref{SPDE}--\eqref{lognormal}.
 
\begin{lemma}\label{lemma:|partial^{bs}u(by)|_V<}
	Assume that there exist a positive sequence $\brho=(\rho_j)_{j\in \NN}$ and a positive number $\kappa$
	satisfying 
	\begin{equation}\label{<kappa}
		\Bigg\|\sum_{j \in \NN} \rho_j |\psi_j | \Bigg\|_{L_\infty(D)} 
		\leq 
		\kappa < \frac{\pi}{2}\,.
	\end{equation}
	Then we have the following.
		Let $\by \in  \RRi$ with
		$b(\by)\in L_\infty(D)$ and $\bs\in \FF$ such that
		$\supp(\bs)\subseteq \supp(\brho)$.  Then we have
		\begin{equation*}
			\|\partial^{\bs}u(\by)\|_V 
			\leq
			C_0\frac{\bs!}{\brho^\bs}
			\exp\big( \|b(\by)\|_{L_\infty(D)} \big), \ \ \by \in \RRi,
		\end{equation*}
		where $C_0=e^\kappa (\cos\kappa)^{-1}\|f\|_{V'}$.
\end{lemma}

In what follows, throughout Sections~\ref{Analyticity and sparsity  for parametric PDEs} and \ref{Semi-discrete parametric approximations}, we denote by $u$ or   $u(\by)$  the weak solution  to parametric equation with log-gamma random inputs \eqref{SPDE}--\eqref{lognormal}($\operatorname{\Gamma}$).

If
$b(\by)\in L_\infty(D)$ with $\by \in \RRip$,
we have the estimate
\begin{equation} \label{eq:uApriori} \|u(\by)\|_V \leq \|f\|_{{V'}} \|
	a(\by)^{-1}\|_{L_\infty(D)} \leq \exp\big(\|b(\by)\|_{L_\infty(D)}\big)\|f\|_{{V'}}.
\end{equation}
It follows from the a-priori estimate \eqref{eq:uApriori} that for
$f\in V'$ the parametric elliptic diffusion problem equation \eqref{SPDE}--\eqref{lognormal}{\rm (L)} admits a
unique solution for parameters $\by$ in the set
\begin{equation}\label{U_0}
	U_0 := \big\{ \by \in \RRip: b(\by)\in L_\infty(D)\big\}.
\end{equation}
%

\begin{lemma}\label{lemma:EE}
	Assume that	for every $j\in \NN$, $\psi_j\in L_\infty(D)$, and there
	exists a positive sequence $(\rho_j)_{j\in \NN}$ such that 
	$\big(\exp(-\rho_j)\big)_{j\in \NN}\in \ell_1(\NN)$ and the
	series $\sum_{j\in \NN}\rho_j|\psi_j|$ converges in
	$L_\infty(D)$.
	Then the set $U_0$ has full measure, i.e.,
	$
	\blambda(U_0) = 1,
	$
	and	contains  all $\by \in \RRip$ with $|\by|_0 < \infty$.
	\end{lemma}

\begin{proof}The proof of this lemma is similar to the proof of  \cite[Theorem 2.2]{BCDM17}. We omit it.
	\hfill
\end{proof}

It is known that, the Laguerre polynomials $\big(L_k\big)_{k\in \NN_0}$ are nontrivial solutions to
the Laguerre's  differential equation of second order
\begin{equation} \label{Dd1} 
	\Dd u = k u, \quad \text{where} \quad \Dd   := - y  \frac{\rd^2  }{\rd y^2} - (a - y)\frac{\rd  }{\rd y}
\end{equation}
(see, e.g., \cite[5.1.3]{Szego1939}). This means that 
\begin{equation*} \label{bb} 
	\Dd   L_k   = k L_k,  \ \ k \in \NN_0. 
\end{equation*}
Observe that, it holds the equality 
\begin{equation} \label{Dd2} 
	\Dd u = - e^y y^{1 - a} \frac{\rd }{\rd  y} \brac{e^{-y} y^{a} \frac{\rd u}{\rd y}},
\end{equation}
which can be shown by taking differentiation of  the function $e^{-y} y^{a} \frac{\rd u}{\rd y}$.  

For a given finite set $J \subset \NN$ and $r \in\NN$, we define the product differential operator 
 \begin{equation} \label{Dd^r_J} 
 	\Dd^r_J  := \brac{\Dd_J}^r, \quad 
 	\Dd_J  := (-1)^{|J|} \prod_{j \in J} e^{y_j}y_j^{1 - a} \frac{\rd}{\rd y_j} \brac{e^{-y_j} y^{a}_j \frac{\rd}{\rd y_j}}.
 	\end{equation}
Notice that there exist polynomials $p_j(t)$, $j = 1,...,2r$, of degree at most $r$  such that 
\begin{equation*} \label{Dd^r} 
	\Dd^r 
	=
	\brac{ - y  \frac{\rd^2 }{\rd y^2} - (a - y)\frac{\rd }{\rd y}}^r
	= \sum_{j=1}^{2r}	p_j(y) \frac{\rd^j }{\rd y^j},
\end{equation*}
and the coefficients of $p_j(y)$ depend on $r$ and $a$ only.	Hence, there exists a constant $C_{a,r}$ such that for $y \ge 0$,
\begin{equation} \label{p_j} 
|p_j(y)|  \le C_{a,r}(1 + y)^r, \ \ j = 1,\ldots,2r.
\end{equation}

	For a finite set $J \subset \NN$, $r \in \NN$ and $\bs \in \FF$, we define
\begin{equation} \label{nu_J}
\nu_{J,\bs}:= \prod_{j \in J} s_j.
\end{equation}		
\begin{lemma} \label{lemma:Dd^r_J v =Laguerre}
	If $v \in L_2(\RRip,V;\blambda)$ and $\Dd_J^r v \in L_2(\RRip,V;\blambda)$, then we have that
 \begin{equation} \label{identity1} 
\Dd^r_J v = \sum_{\bs \in \FF} \nu_{J,\bs}^r v_\bs L_\bs.
\end{equation}
\end{lemma}

\begin{proof}
We first prove that  if $v \in L_2(\RR_+,V;\lambda)$ and $\Dd v \in L_2(\RR_+,V;\lambda)$, then we have that
\begin{equation} \label{identity2} 
	\Dd v = \sum_{s\in \NN_0} s v_s L_s.
\end{equation}
Indeed, by  \eqref{Dd1}, \eqref{Dd2} and integration by parts twice, we have that,
\begin{equation*}
	\begin{aligned}
		s  v_s &:= \ \int_{\RRp} v(y)\, \brac{sL_s(y)}\, \rd\lambda(y)
		=  \, 	\int_{\RRp}v(y) \Dd L_s(y) \,\rd y
	\\
	&	=  \, -	\int_{\RRp}v(y)  \frac{\rd }{\rd y} \brac{e^{-y} y^{a} \frac{\rd L_s(y)}{\rd y}}\,\rd y		 
	 =  \, \int_{\RRp} \frac{\rd v(y) }{\rd y} e^{-y} y^{a} \frac{\rd L_s(y)}{\rd y}\,\rd y \\[1ex]
		&=  \ 	- \int_{\RRp}  \frac{\rd }{\rd y} \brac{e^{-y} y^{a} \frac{\rd v(y)}{\rd y}} L_s(y) \,\rd y 
		 =  \ 	\int_{\RRp}e^{-y}y^{a-1}\Dd v(y) L_s(y)\,\rd y\\[1ex]	
		&=  \ 	\int_{\RRp} \Dd v(y) L_s(y)\,\rd \lambda(y), \\[1ex]					
		\end{aligned}
\end{equation*}
which implies \eqref{identity2}.  By using  the tensor product argument, \eqref{identity2}, definition \eqref{Dd^r_J}  one can easily deduce \eqref{identity1}.
\hfill
\end{proof}

For $\theta, \lambda \ge 0$, we define the set 
$\bp(\theta, \lambda):= \brac{p_\bs(\theta, \lambda) }_{\bs \in \FF}$ by 
\begin{equation*} \label{[p_s]}
	p_\bs(\theta, \lambda) := \prod_{j \in \NN} (1 + \lambda s_j)^\theta, \quad \bs \in \FF.
\end{equation*}

We often use the following inequality for fixed $\theta$ and $\lambda$
	\begin{equation}\label{ineq-theta,lambda}
		(1+ \lambda k)^{\theta} \le C_{\theta,\lambda} k^{\theta}, \ \  k \in \NN,
	\end{equation}
with some constant $C_{\theta,\lambda}$	depending on $\theta,\lambda$ only. For fixed $0<p<\infty$ and $\theta \ge 0$, let  $r_{p,\theta} \in \NN$ be chosen such that 
	\begin{equation}\label{C_{p,theta}}
p(r_{p,\theta}-\theta)>1, \ \ \ \text{and \ denote} \ \ \ 
C_{p,\theta}:= \sum_{k \in \NN} k^{-p(r_{p,\theta}-\theta)}.
\end{equation}
	For a finite subset $J$ of $\NN$ and $r \in \NN$, we define 
\begin{equation}\label{A_r(J):=}
	A_r(J):= \Bigg(\int_{\RRip}\prod_{j \in J} (1 + y_j)^{2r} 
	\exp\Big(2 \|b(\by)\|_{L_\infty(D)}\Big) \rd \blambda(\by)\Bigg)^{1/2}.
\end{equation}

	\begin{theorem} \label{thm: summability-ell_p-rho}
		Let  $0< p < \infty$, $\brho=(\rho_j)_{j\in \NN}$ be a positive sequence satisfying  condition~\eqref{<kappa}, and $\brho^{-1}\in \ell_p(\NN)$. 
	Assume that  for any $r\in \NN$, there exist a constant $K_r$ such that
	\begin{equation}\label{A_r(J)<}
		A_r(J) \le K_r^{|J|}
	\end{equation}
for any  finite set $J\subset \NN$.	 
	Then $\brac{\norm{u_\bs}{V}}_{\bs\in \FF}\in \ell_p(\FF)$.	
	
	Moreover, if in addition, $p < 2$, for any fixed $\theta, \lambda \ge 0$,  we can construct a set 
	$\bsigma=(\sigma_\bs)_{\bs \in \FF}$  with positive $\sigma_\bs$, and a constant $M$ such that	
	\begin{equation} \label{ell_2-summability-rho}
		\left(\sum_{\bs\in\FF} (\sigma_\bs \|u_\bs\|_V)^2\right)^{1/2} \ \le M^{1/2} \ <\infty, \ \ \text{with} \ \
		\norm{\bp(\theta,\lambda)\bsigma^{-1}}{\ell_q(\FF)} \le M^{1/q} < \infty,
	\end{equation}	
	where $q := 2p/(2-p)$.  The set $\bsigma$ and constant $M$ depend on $a,  p, \theta, \lambda, \kappa$ only. 
	\end{theorem}
	
		\begin{proof} 
		We have by  Lemma~\ref{lemma:Dd^r_J v =Laguerre} and the Parseval's  identity \eqref{Parseval's identity}
		\begin{equation} \label{P-identity} 
			\int_{\RRip}\norm{\Dd^r_J u(\by)}{V}^2 \rd \blambda(\by) 
			=\sum_{\bs \in \FF} \nu_{J,\bs}^{2r}\norm{u_\bs}{V}^2.
		\end{equation}		
From \eqref{p_j}, we get for $\by \in \RRip$,	
		\begin{equation*} 
			\begin{aligned}
			\norm{\Dd^r_J u(\by)}{V}
			&=
			\norm{\prod_{j \in J}\brac{ - y_j  \frac{\rd^2 }{\rd y_j^2} - (a - y_j)\frac{\rd }{\rd y_j}}^r u(\by)}{V}
			\\
			&\le 
		C_{a,r}^{|J|} \prod_{j \in J} (1 + y_j)^r 
			\sum_{\supp (\bk) = J, \ |\bk|_\infty \le 2r} \norm{\partial^\bk u(\by)}{V},
		\end{aligned}
		\end{equation*}
		where $C_{a,r}$ is the constant as in \eqref{p_j}.
		Applying Lemma \ref{lemma:|partial^{bs}u(by)|_V<} gives
		\begin{equation*} 
			\norm{\Dd^r_J u(\by)}{V}
			\le 
		C_{a,r}^{|J|} \prod_{j \in J} (1 + y_j)^r
			\sum_{\supp (\bk) = J, \ |\bk|_\infty \le 2r} C_0\frac{\bk!}{\brho^\bk}
			\exp\big( \|b(\by)\|_{L_\infty(D)} \big).
		\end{equation*}
		Hence,	
		\begin{equation*}
			\int_{\RRip}\norm{\Dd^r_J u(\by)}{V}^2 \rd \blambda(\by) 
			\le 
		A_r(J)^2	C_0^2 	C_{a,r}^{2|J|} 
			\Bigg( \sum_{\supp (\bk) = J, \ |\bk|_\infty \le 2r} \frac{\bk!}{\brho^\bk}\Bigg)^2.
		\end{equation*}
Moreover, $\bk! \le ((2r)!)^{|J|}$ when  $\supp (\bk) = J, \ |\bk|_\infty \le 2r$.
		Therefore, for $\bs \in \FF$ with $\supp (\bs)=J$, we have $\nu_{J,\bs}^{2r} = \bs^{2r}$ and by \eqref{P-identity} and 
		\eqref{A_r(J)<}, 
\begin{equation} \label{norm{u_bs}{V}}
\begin{aligned}
				\norm{u_\bs}{V}	
&	\le 
	\bs^{-r}	\Bigg(\int_{\RRip}\norm{\Dd^r_J u(\by)}{V}^2 \rd \blambda(\by) \Bigg)^{1/2}
	\\
	&
	\le 
	\bs^{-r} A_r(J)	C_0 	C_{a,r}^{|J|} ((2r)!)^{|J|}
	\sum_{\supp (\bk) = \supp(\bs), \ |\bk|_\infty \le 2r} \brho^{-\bk}
	\\
&
\le 
C_0 C_1^{|\supp(\bs)|}\bs^{-r}
\sum_{\supp (\bk) = \supp(\bs), \ |\bk|_\infty \le 2r} \brho^{-\bk}=: \beta_\bs,
\end{aligned}
\end{equation}
where $C_1:= K_rC_{a,r}(2r)!$ and $K_r$ is as in \eqref{A_r(J)<}.

Let $\bbeta=(\beta_\bs)_{\bs \in \FF}$. In the following, for any $\theta',\lambda\geq 0$ we will prove that there exists $r \in \NN$ depending on $\theta'$ and $p$ only and a constant $M$ depending on 
$a,  p, \theta', \lambda,\kappa$ only such that
\begin{equation} \label{norm1}
\norm{\bp(\theta',\lambda)\bbeta}{\ell_p(\FF)}
\ \le \
M^{1/p}.
\end{equation}
By \eqref{ineq-theta,lambda}  we have
$$
p_\bs(\theta',\lambda) \beta_\bs 	\le 
C_0\bs^{-r+\theta'} (C_{\theta',\lambda} C_1)^{|\supp(\bs)|}
\sum_{\supp (\bk) = \supp(\bs), \ |\bk|_\infty \le 2r} \brho^{-\bk},
$$
where $C_{\theta',\lambda}$	 is as in 	\eqref{ineq-theta,lambda}. 
Hence,
		\begin{equation} \label{sum2}
			\sum_{\supp (\bs) = J}	\big(p_\bs(\theta',\lambda) \beta_\bs\big)^p	
			\leq 
			C_0(C_{\theta',\lambda} C_1)^{p|J|}	\sum_{\supp (\bs) = J} \bs^{-p(r-\theta')} 
			\brac{\sum_{\supp (\bk) = J, \ |\bk|_\infty \le 2r} \brho^{-\bk}}^p.
		\end{equation}
 Let  $r= r_{p,\theta'}$ be chosen and $C_{p,\theta'}$ be defined as in \eqref{C_{p,theta}}.
		Then we have 
		\begin{equation} \label{sum2}
		\sum_{\supp (\bs) = J} \bs^{-p(r-\theta')} 
			=	
			\prod_{j \in J}	\sum_{s_j \in \NN} s_j^{-p(r-\theta')} 
			\le
			C_{p,\theta'}^{|J|}.
			\end{equation}	
		Letting $|J| := m$ and $J:= \brab{j_1,\ldots,j_m}$, we get 
		\begin{equation*}
			\brac{\sum_{\supp (\bk) = J, \ |\bk|_\infty \le 2r} \brho^{-\bk}}^p		
			= 
			\prod_{i=1}^m \brac{\sum_{\ell=1}^{2r}\brac{\rho_{j_i}}^\ell}^p.		
		\end{equation*}
Hence,
			\begin{equation*}
				\sum_{\supp (\bs) = J}	\big(p_\bs(\theta',\lambda) \beta_\bs\big)^p	
				\le 
				C_0(C_{\theta',\lambda} C_1C_{p,\theta'})^m	\prod_{i=1}^m \brac{{\sum_{\ell=1}^{2r}}\brac{\rho_{j_i}}^\ell}^p
				=	
				C_0\prod_{i=1}^m K \brac{\sum_{\ell=1}^{2r}\brac{\rho_{j_i}}^\ell}^p,
			\end{equation*}
			where 
			$$
			K=C_{\theta',\lambda}K_rC_{a,r}C_{p,\theta'}(2r)!.
			$$ It follows that
			\begin{equation*}
				\begin{aligned}
				\sum_{\bs \in \FF}	\big(p_\bs(\theta',\lambda) \beta_\bs\big)^p	
				&
						\le C_0\sum_{m=1}^{\infty}\sum_{j_1,\ldots,j_m=1}^\infty 	
						\prod_{i=1}^m K \brac{\sum_{\ell=1}^{2r}\brac{\rho_{j_i}}^\ell}^p	
					\\
					& =	C_0\prod_{k=1}^\infty 
					\Bigg(1+K \brac{\sum_{\ell=1}^{2r}\brac{\rho_{j_i}}^\ell}^p	\Bigg)
					\\
					&	\leq C_0
						\exp\Bigg(K\sum_{k=1}^\infty  
						\brac{\sum_{\ell=1}^{2r}\brac{\rho_{j_i}}^\ell}^p	\Bigg),
				\end{aligned}
			\end{equation*}
			which is finite since $\brho^{-1} \in \ell_p(\NN)$. This proves \eqref{norm1}.
		Hence, from \eqref{norm{u_bs}{V}} we deduce that $\brac{\norm{u_\bs}{V}}_{\bs \in \FF}\in \ell_p(\FF)$ by choosing $\theta'=0$.

We now prove \eqref{ell_2-summability-rho}. If $0< p <2$, we define the set 
$\bsigma=(\sigma_\bs)_{\bs \in \FF}$ by
\begin{equation*} 
	\sigma_\bs:= 	\beta_\bs^{p/2 - 1},  \ \bs \in \FF. 
\end{equation*}	
Then we have
\begin{equation*}  \label{sum_ (bsigma^bs...e}
	\sum_{\bs\in\FF} (\sigma_\bs \|u_\bs\|_V)^2
	\ \le \
	\sum_{\bs\in\FF} 
	\brac{\beta_\bs^{p/2 - 1}\beta_\bs}^2
	\ = \
	\norm{\bbeta}{\ell_p(\FF)}^p 		
	\ \le \ M,
\end{equation*}	
and by choosing $\theta'=\theta q/p$,
\begin{equation*} 
	\big\|\bp(\theta,\lambda)\bsigma^{-1}\|_{\ell_q(\FF)}^q
	\ = \ 
	\sum_{\bs\in\FF} 
	\brac{\beta_\bs^{1- p/2}}^{2p/(2-p)} p_\bs(\theta,\lambda)^q
	\ = \
	\norm{\bp(\theta',\lambda)\bbeta}{\ell_p(\FF)}^p 
	\ \le \
	M,
\end{equation*}	
which proves \eqref{ell_2-summability-rho}.		
		\hfill
	\end{proof}
	
		\begin{lemma} \label{lemma:A_r(J)}
		Let  $\bb = \brac{b_j}_{j \in \NN}$ be defined by 
		$b_j:= \norm{\psi_j}{L_\infty(D)}$. Assume 
		$\bb\in \ell_1(\NN)$ and
		\begin{equation*}\label{|bb|_infty}
			\norm{\bb}{\ell_\infty(\NN)} = b_0 <\frac{1}{2}.
		\end{equation*}
		Then 
	\begin{equation}\label{A_r(J)<(2)}
		A_r(J) \le \exp \brac{\frac{a}{1-2b_0} \norm{\bb}{\ell_1(\NN)}}K_{a,r,\bb}^{|J|},
	\end{equation}
	where
	\begin{equation} \label{K_{a,r,bb}}
	K_{a,r,\bb}:= \brac{\int_{\RR_+}  
		\frac{(1+y)^{2r} y^{a-1}}{\Gamma(a)} \exp\big(y(2b_0-1)\big) \rd y}^{1/2} <\infty.
\end{equation}
	\end{lemma}
	
	\begin{proof} 
			For any finite set $J\subset \NN$ we have
		\begin{equation*}
			\begin{aligned}
				A_r(J)^2&:=		
				\int_{\RRip}\prod_{j \in J} (1 + y_j)^{2r} 
				\exp\Big(2 \|b(\by)\|_{L_\infty(D)}\Big) \rd \blambda(\by)
				\\
				&\leq  \int_{\RRip}\prod_{j \in J} (1 + y_j)^{2r} 
				\exp\Big(2 \sum_{j\in \NN}b_jy_j\Big) \prod_{j\in \NN}\frac{1}{\Gamma(a)}y_j^{a-1} \exp(-y_j) \rd y_j
				\\
				&\leq \prod_{j\in J} \int_{\RR_+}  \frac{(1+y_j)^{2r} y^{a-1}}{\Gamma(a)} \exp\big(y_j(2b_0-1)\big) \rd y_j \prod_{j\not\in J} \int_{\RR_+}  \frac{y^{a-1}}{\Gamma(a)} \exp\big(y_j(2b_j-1)\big) \rd y_j.
			\end{aligned}
		\end{equation*}
		By changing variable $y_j(1-2b_j)=\xi_j$ and by
		we get
		$$
		A_r(J)^2\leq \brac{K_{a,r,\bb}^{|J|}}^2 \prod_{j\not \in J} \frac{1}{(1-2b_j)^{a}}.
		$$
		We have
		\begin{align*}
			\log\bigg(\prod_{j\not \in J} \frac{1}{(1-2b_j)^{a}}\bigg) 
			&=a\sum_{j\not \in J} \log\frac{1}{1-2b_j} 
			= a\sum_{j\not \in J}  \log \bigg(1+\frac{2b_j}{1-2b_j}\bigg) 
			\\
			&\leq a\sum_{j\not \in J} \frac{2b_j}{1-2b_0} 
			\leq \frac{2a}{1-2b_0} \norm{\bb}{\ell_1(\NN)}.
		\end{align*}
The bound \eqref{A_r(J)<(2)} has been proven.	
	\hfill	
	\end{proof}	
	
	\begin{theorem} \label{thm: summability-ell_p-b-L}
			Let $0< p \le 1$ and  $\bb = \brac{b_j}_{j \in \NN}$ be defined by	$b_j:= \norm{\psi_j}{L_\infty(D)}$.  Let $r=r_{p,0}\in \NN$ be a fixed number satisfying \eqref{C_{p,theta}}.
 Assume that
	$\bb\in \ell_p(\NN)$,
\begin{equation*}\label{eq:max<1/2-1}
	\norm{\bb}{\ell_\infty(\NN)} = b_0 <\frac{1}{2},
\end{equation*}
and
\begin{equation} \label{norm2} \
\norm{\bb}{\ell_1(\NN)}	< K^{-1},
\end{equation}
where 
	\begin{equation*}  
K:= e (2r)!C_{a,r}C_{p,0}K_{a,r,\bb},
\end{equation*}
	and the constants $C_{a,r}$,  $C_{p,0}$,  $K_{a,r,\bb}$, 
	are defined as in \eqref{p_j},   \eqref{C_{p,theta}}, \eqref{K_{a,r,bb}}, respectively. 	 
	Then $\brac{\norm{u_\bs}{V}}_{\bs \in \FF}\in \ell_p(\FF)$. 	
	
	Moreover, for any fixed $\theta, \lambda \ge 0$, if  $q := 2p/(2-p)$,  $\theta'=2\theta /(2-p)$, $r= r_{p,\theta'}$ is a fixed number \eqref{C_{p,theta}} and $K$ in \eqref{norm2} is replaced by 
		\begin{equation} \label{K}
		K:= e(2r)!C_{a,r}C_{p,\theta'}C_{\theta',\lambda}K_{a,r,\bb}
	\end{equation}
with $C_{\theta',\lambda}$ being defined as in \eqref{ineq-theta,lambda}, then	we can construct a set 
	$\bsigma=(\sigma_\bs)_{\bs \in \FF}$, and  a constant $M$  such that	
	\begin{equation} \label{ell_2-summability-L}
		\left(\sum_{\bs\in\FF} (\sigma_\bs \|u_\bs\|_V)^2\right)^{1/2} \ \le M^{1/2} \ <\infty, \ \ \text{with} \ \
		\norm{\bp(\theta,\lambda)\bsigma^{-1}}{\ell_q(\FF)} \le M^{1/q} < \infty.
	\end{equation}	
	The set $\bsigma$ and constant $M$ depend on $a, \bb, p, \theta, \lambda$ only. 
\end{theorem}

\begin{proof} 
	We define	the sequence $\brho_{\bk} = \brac{\rho_{\bk,j}}_{j \in \NN}$ depending on $\bk \in \FF$, by 
	\begin{equation*}  
		\rho_{\bk,j}:= 
\begin{cases}
\frac{k_j}{b_j|\bk|_1} \ \ &\text{if}\ \ j\in \supp(\bk), \\
0\ \ &\text{if}\ \ j\not \in \supp(\bk).
\end{cases}		
\end{equation*}		
Notice that
		\begin{equation*}\label{<1}
	\sup_{\bk \in \FF}	\Bigg\|\sum_{j \in \NN} \rho_{\bk,j} |\psi_j | \Bigg\|_{L_\infty(D)} 
		\leq 1.
	\end{equation*}
	Applying Lemma \ref{lemma:|partial^{bs}u(by)|_V<}  gives for $\by \in \RRip$  and $\bk\in \FF$,
	\begin{equation*}  
\begin{aligned}
	\norm{\partial^\bk u(\by)}{V} &
\le 
C_0\frac{\bk!}{\brho_\bk^\bk}
\exp\big( \|b(\by)\|_{L_\infty(D)} \big)
=
C_0\frac{\bk!\bb^\bk |\bk|^{|\bk|}}{\bk^{\bk}}
\exp\big( \|b(\by)\|_{L_\infty(D)} \big).
\end{aligned}
	\end{equation*}
	Similarly to the proof of the previous theorem we can derive	
\begin{equation*}
\begin{aligned}
\int_{\RRip}\norm{\Dd^r_J u(\by)}{V}^2 \rd \blambda(\by) 
&\le 
A_r(J)^2	C_0^2 	C_{a,r}^{2|J|} 
\Bigg( \sum_{\supp (\bk)=J, \ |\bk|_\infty \le 2r} 
\frac{\bk!\bb^\bk |\bk|^{|\bk|}}{\bk^{\bk}}\Bigg)^2,
\end{aligned}
\end{equation*}
where $C_{a,r}$ is the constant as in \eqref{p_j} and $A_r(J)$ as in \eqref{A_r(J):=}.
By Lemma~\ref{lemma:A_r(J)} and the inequalities
$\frac{|\bk|^{|\bk|}}{\bk^{\bk}}\le \frac{|\bk|! e^{|\bk|}}{\bk!}$ and $\bk!\leq ((2r)!)^{|J|}$ with $\supp(\bk)=J$ and $|\bk|_\infty \le 2r$,
we have
\begin{equation*}
	\begin{aligned}
		\int_{\RRip}\norm{\Dd^r_J u(\by)}{V}^2 \rd \blambda(\by) 
		&\leq 	C_1^2K_{a,r,\bb}^{2|J|}		C_{a,r}^{2|J|} 
		((2r)!)^{2|J|}\Bigg( \sum_{\supp (\bk)=J, \ |\bk|_\infty \le 2r} \frac{(e\bb)^\bk |\bk|!}{\bk!}\Bigg)^2,
	\end{aligned}
\end{equation*}
where
	\begin{equation*} \label{}
C_1:= C_0 \exp \brac{\frac{a}{1-2b_0}\norm{\bb}{\ell_1(\NN)}}. 
\end{equation*}
This together with  and \eqref{P-identity} implies for $\bs \in \FF$ with $\supp (\bs)=J$,
		\begin{equation} \label{norm{u_bs}}
		\norm{u_\bs}{V}		
		\le 
	C_1 C_2^{|J|}	\bs^{-r} \sum_{\supp (\bk)=J, \ |\bk|_\infty \le 2r} \frac{(e\bb)^\bk |\bk|!}{\bk!}.
	\end{equation}
	where
	\begin{equation*} \label{}
		C_2:= (2r)!K_{a,r,\bb}C_{a,r}. 
	\end{equation*}
Denote by $\beta_\bs$ the right-hand side of \eqref{norm{u_bs}} and let
$\bbeta=(\beta_\bs)_{\bs \in \FF}$.  For any $\theta',\lambda\geq 0$ we will prove that there exists a constant 
$M$ depending on $a, \bb, p, \theta, \lambda$ only	 such that  
\begin{equation} \label{norm-1}
\norm{\bp(\theta,\lambda)\bbeta}{\ell_p(\FF)}
\ \le \
M^{1/p}.
\end{equation}
By using the inequality \eqref{ineq-theta,lambda},  we get
	\begin{equation*} \label{sum1}
\sum_{\supp (\bs) = J}	\big(p_\bs(\theta',\lambda) \beta_\bs \big)^p
		\le 
C_1(C_{\theta',\lambda}C_2)^{p|J|} \sum_{\supp (\bs) = J}	\bs^{-p(r-\theta')}	
		\brac{\sum_{\supp (\bk) = J, \ |\bk|_\infty \le 2r} \frac{(e\bb)^\bk |\bk|!}{\bk!}}^p,
	\end{equation*}
	where $C_{\theta',\lambda}$ is as in \eqref{ineq-theta,lambda}.
 Let  $r= r_{p,\theta'}$ be chosen and $C_{p,\theta'}$ be defined as in \eqref{C_{p,theta}}. 
By \eqref{sum2} we get
		\begin{equation*}
	\begin{aligned}
\sum_{\supp (\bs) = J}	\big(p_\bs(\theta',\lambda) \beta_\bs \big)^p		
&\le 
C_1C_3^{p|J|}
\sum_{\supp (\bk) = J, \ |\bk|_\infty \le 2r} \brac{\frac{(e\bb)^\bk |\bk|!}{\bk!}}^p
\\
& = \,
C_1\sum_{\supp (\bk) = J, \ |\bk|_\infty \le 2r} 	\brac{\frac{(eC_3\bb)^\bk |\bk|!}{\bk!}}^p,		 
	\end{aligned}
	\end{equation*}
		where $C_3:=e^{-1}K$ and $K$ is as in \eqref{K}.
Then we obtain
	\begin{equation} \label{sum3}
	\sum_{\bs\in \FF}	p_\bs(\theta',\lambda)^p \beta_\bs^p	
\le  
	C_1\sum_{\bk \in \FF: \, |\bk|_\infty \le 2r} 
\brac{\frac{\bar{\bb}^\bk |\bk|!}{\bk!}}^p, 
\end{equation}
where $\bar{\bb}:= K\bb$. 
 We have by \eqref{norm-1} that $\norm{\bar{\bb}}{\ell_p(\NN)} < \infty$ and $\norm{\bar{\bb}}{\ell_1(\NN)} < 1$.
Hence, from \cite[Theorem 7.2]{CDS10} and \eqref{sum3} we derive \eqref{norm2}, and therefore, 
$\norm{\bp(\theta,\lambda)\bbeta}{\ell_p(\FF)}\ \le \ M^{1/p}$.  It follows from \eqref{norm{u_bs}} that $\brac{\norm{u_\bs}{V}}_{\bs \in \FF}\in \ell_p(\FF)$ by choosing $\theta'=0$.

The  proof of \eqref{ell_2-summability-L} is similar to that of \eqref{ell_2-summability-rho} in Theorem~\ref{thm: summability-ell_p-rho}.
	\hfill
\end{proof}

\begin{remark} \label{remark2.1}
{\rm	
We compare our  method in establishing the sparsity for Laguerre GPC expansion coefficients with  the method  for Hermite GPC expansion coefficients presented in \cite[Section 3.6]{DNSZ2023}. 
Both  methods are based on the bounds of parametric partial derivatives as in Lemma~\ref{lemma:|partial^{bs}u(by)|_V<}. They both also employ  the criterion for $\ell_p$-summability as in \cite[Theorem 7.2]{CDS10}. 

A notable difference distinguishes the two methods. The method in \cite[Section 3.6]{DNSZ2023} relies on an identity between an weighted square sum of the energy norms of Hermite GPC expansion coefficients and weighted square sum of  the $L_2(\RRi,V; \bgamma)$-norms of the parametric partial derivatives \cite[Theorem 3.3]{BCDM17}. This equality is established by  combining Parseval's identity and
Rodrigues' formula, tool sets that do not extend to the Laguerre setting. 
By contrast, our method employs the identity \eqref{identity1} from Lemma~\ref{lemma:Dd^r_J v =Laguerre} which expresses the derivative  the derivative $\Dd_J^r v$ in terms of Laguerre GPC expansion coefficients of $v$. This representation  is derived from  the Laguerre's  second order differential equation satisfied by the Laguerre polynomials.
}
\end{remark}

\section{Semi-discrete  approximations}
\label{Semi-discrete parametric approximations}
In this section, from the weighted $\ell_2$-summability results of 
Theorems~\ref{thm: summability-ell_p-rho} and \ref{thm: summability-ell_p-b-L}, we  derive convergence rates for the semi-discrete linear parametric-variable approximations of the solution $u(\by)$ to the parametric equation with with log-gamma random inputs  \eqref{SPDE}--\eqref{lognormal}($\operatorname{\Gamma}$) by sparse-grid polynomial interpolations, extended least-squares sampling algorithms and the associated semi-discrete quadratures. 

\subsection{ Sparse-grid polynomial  interpolation}
\label{ Hermite gpc interpolation approximation}
In this section, we construct sparse-grid polynomial interpolations for semi-discrete parametric approximation of the solution $u(\by)$ to the parametric equation \eqref{SPDE}--\eqref{lognormal}{\rm ($\operatorname{\Gamma}$)}.
Observe that under the assumption of Lemma \ref{lemma:EE},
Notice that $u(\by)$ is well-defined for every $\by\in U_0$, where $U_0$ is the set defined as 
in \eqref{U_0}. By Lemma \ref{lemma:EE}, the set $U_0$ has full measure, i.e., $\blambda(U_0)=1$, and	contains  all $\by \in \RRip$ with $|\by|_0 < \infty$, where $|\by|_0$ denotes the number of nonzero components $y_j$ of $\by$. Moreover, $u(\by)$ can be treated as a representative of an element in 
$L_2(\RRip,V;\blambda)$.

For $m \in \NN_0$, let $Y_m = (y_{m;k})_{k=1}^m$ be the increasing sequence of  the $m$ roots of the Laguerre polynomial $L_m$, ordered as
$
0 < y_{m;1} < \ldots < y_{m;m}.
$
We use also the convention $Y_0 = (y_{0;0})$ with $y_{0;0} = 0$.

For  a function $v$  on $\RRp$ taking values in a Hilbert space $V$ and $m \in \NN_0$, we define the   Lagrange interpolation operator $I_m$  by
\begin{equation} \label{I_ v}
	I_m v:= \ \sum_{k=1}^m v(y_{m;k}) \ell_{m;k}, \quad 
	\ell_{m;k}(y) := \prod_{1 \le j \le m, \ j\not=k}\frac{y - y_{m;j}}{y_{m;k} - y_{m;j}},
\end{equation}		
(in particular, $I_0 v = v(y_{0;0})\ell_{0;0}(y)= v(0)$ and $\ell_{0;0}(y)=1$). Notice that $I_m v$ is a function on $\RRp$ taking values in $V$ and  interpolating $v$ at $y_{m;k}$, i.e., $I_m v(y_{m;k}) = v(y_{m;k})$. 

Let
\begin{equation*} \label{lambda_m}
	\lambda_m:= \ \sup_{\big\|v\sqrt{l_a}\big\|_{L_\infty(\RR_+)} \le 1} 
	\big\|(I_m v)\sqrt{l_a}\big\|_{L_\infty(\RR_+)} 
\end{equation*}	
be the Lebesgue constant, where $l_a$ is given as in \eqref{l_a}. 
It was proven in \cite{MO2001} that
\begin{equation} \nonumber
	\lambda_m
	\ \approx \
	Cm^{1/6}, \quad m \in \NN,
\end{equation}
for some positive constant $C$ independent of $m$.


%
%

For  a function $v$  on $\RRip$ taking values in a Hilbert space $V$ and $\bs \in \FF$, 
 we introduce the tensor product operator $\Delta_\bs$, $\bs \in \FF$, by
\begin{equation*} \label{Delta_bs v}
	\Delta_\bs v
	:= \
	\bigotimes_{j \in \NN} \Delta_{s_j} v, \quad 
	\Delta_m
	:= \
	I_m - I_{m-1}, \ \ I_{-1} = 0 \ \ (m \in \NN_0),
\end{equation*}
where the  operator
$\Delta_{s_j}$ is successively applied to the univariate function $\bigotimes_{i<j} \Delta_{s_i} v$ by considering it as a 
function of  variable $y_j$ with the other variables held fixed.
We define for $\bs \in \FF$,
\begin{equation} \nonumber
	\ell_{\bs;\bk}
	:= \
	\bigotimes_{j \in \NN} \ell_{s_j;k_j}, \quad
	\pi_\bs
	:= \
	\prod_{j \in \NN} \pi_{s_j}.
\end{equation}
For $\bs \in \FF$ and $\bone \le \bk \le\bs$, let $E_\bs$ be the subset in $\FF$ of all $\be$ such that $e_j$ is either $1$ or $0$ if $s_j > 0$, and $e_j$ is $0$ if $s_j = 0$, and let $\by_{\bs;\bk}:= (y_{s_j;k_j})_{j \in \NN} \in \RRip$. Here, the inequities $\bone \le \bk \le\bs$ mean by convention that $1 \le k_j \le s_j$ for $j \in \supp (\bs)$ and 
$\supp (\bk) = \supp (\bs)$.  Recall that $|\bs|_1 := \sum_{j \in \NN} s_j$ for  $\bs \in \FF$. It is easy to check that the interpolation operator $\Delta_\bs$ can be represented in the form
\begin{equation} \label{Delta_bs=}
	\Delta_\bs v				
	\ = \
	\sum_{\be \in E_\bs} (-1)^{|\be|_1} \sum_{\bone \le \bk \le \bs - \be} v(\by_{\bs - \be;\bk}) \ell_{\bs - \be;\bk}.
\end{equation}

Let $0 < q < \infty$ and $\bsigma= (\sigma_\bs)_{\bs \in \FF}$ be a set of positive numbers.
Let the set $\Lambda(\xi)$  for $\xi >1$ be defined by
\begin{equation} \label{Lambda(xi):=}
	\Lambda(\xi)
	:= \ 
	\ \big\{\bs \in \FF: \, \sigma_\bs\le \xi^{1/q}  \big\}.
\end{equation}
We introduce the sparse-grid polynomial  interpolation operator $I_{\Lambda(\xi)}$  by
\begin{equation*} \label{I_Lambda}
	I_{\Lambda(\xi)}
	:= \
	\sum_{\bs \in \Lambda(\xi)} \Delta_\bs.
\end{equation*}
 By the formula \eqref{Delta_bs=}  we  can  represent $I_{\Lambda(\xi)}$ in the form
\begin{equation} \label{I_Lambda(xi)=}
	I_{\Lambda(\xi)} v			
	\ = \
	\sum_{(\bs,\be,\bk) \in G(\xi)}  (-1)^{|\be|_1} v(\by_{\bs - \be;\bk})\ell_{\bs - \be;\bk},
\end{equation}
where	
\begin{equation*} \label{G(xi):=}
	G(\xi)				
	:= \
	\{(\bs,\be,\bk) \in \FF \times \FF \times \FF: 
	\, \bs \in \Lambda(\xi), \ \be \in E_\bs, \ \bone \le \bk \le \bs - \be \}.
\end{equation*}
Note $I_{\Lambda(\xi)} v$ is determined by the values of $v$ at the points $\by_{\bs - \be;\bk}$, 
$(\bs,\be,\bk) \in G(\xi)$, and the number of these points is $|G(\xi)|$. Moreover, 
$|\by_{\bs - \be;\bk}|_0 < \infty$, and, consequently  $\by_{\bs - \be;\bk} \in U_0$ for every $(\bs,\be,\bk) \in G(\xi)$. Hence, for the solution $u$ to the parametric elliptic PDE \eqref{SPDE}--\eqref{lognormal}{\rm ($\operatorname{\Gamma}$)}, the function $I_{\Lambda(\xi)} u$ is well-defined.

\begin{theorem}\label{thm:coll-approx-X}
	Let the assumptions of either Theorem~\ref{thm: summability-ell_p-rho} or  
	Theorem~\ref{thm: summability-ell_p-b-L} for  $0<p<1$. 		Let $\Lambda(\xi)$ be the set defined in \eqref{Lambda(xi):=} for the set $\bsigma$  in these theorems satisfying \eqref{ell_2-summability-rho} or \eqref{ell_2-summability}, respectively. Then there exists a constant $C$ such that for each $n > 1$, we can construct a sequence of points  
		$(\by_{\bs - \be;\bk})_{(\bs,\be,\bk) \in G(\xi_n)}$ so that
		$|G(\xi_n)| \le n$ and
		\begin{equation*} \label{u-I_Lambda u}
			\|u -I_{\Lambda(\xi_n)}u\|_{L_2(\RRip,V;\blambda)} 
			\leq C n^{-(1/p - 1)}.
		\end{equation*}
\end{theorem}

\begin{proof}
	This theorem can be proven similarly to \cite[Corollary 3.1]{Dung21} for functions in the Bochner space
	$L_2(\RRi,X;\bgamma)$. 
	\hfill
\end{proof}

\subsection{Sparse-grid quadrature for numerical integration}

If $v$ is a function defined on $\RR$ taking values in a Hilbert space $X$,  the function  $I_m v$ in  \eqref{I_ v}  generates the quadrature formula  which is defined by
\begin{equation} \nonumber
	Q_m v
	:= \ \int_{\RR_+} I_m v(y) \, \rd \lambda(y)
	\ = \
	\sum_{k=0}^m\omega_{m;k}\, v(y_{m;k}), \quad
	\omega_{m;k}
	:=  \int_{\RR_+} \ell_{m;k}(y) \, \rd \lambda(y).
\end{equation}
For a function $v$ defined on $\RRip$ taking value in $X$, we introduce the operator $\Delta^{{\rm Q}}_\bs$ defined for $\bs \in \FF$ by
\begin{equation} \nonumber
	\Delta^{{\rm Q}}_\bs v
	:= \
	\bigotimes_{j \in \NN} \Delta^{{\rm Q}}_{s_j} v, \quad 	
	\Delta^{{\rm Q}}_m
	:= \
	Q_m - Q_{m-1},  \ \ Q_{-1} := 0 \ \ (m \in \NN),
\end{equation} 
where the univariate operator
$\Delta^{{\rm Q}}_{s_j}$ is applied to the univariate function 
$\bigotimes_{i<j} \Delta^{{\rm Q}}_{s_i} v$ by considering it as a 
function of  variable $y_i$ with the other variables held fixed.  
For a finite set $\Lambda \subset \FF$, we introduce the quadrature operator 
$Q_\Lambda$ by
\begin{equation} \nonumber
	Q_\Lambda v
	:= \
	\sum_{\bs \in \Lambda} \Delta^{{\rm Q}}_\bs  v.
\end{equation} 
Further, if $\phi \in V'$ is a bounded linear functional on $V$, denote by $\langle \phi, v \rangle$ the value of $\phi$ in $v$.  

By the formula \eqref{I_Lambda(xi)=} we  can  represent the operator $Q_{\Lambda(\xi)}$ in the form
\begin{equation*} \label{Q_Lambda_rev(xi)=}
	Q_{\Lambda(\xi)}v			
	\ = \
	\sum_{(\bs,\be,\bk) \in G(\xi)}  (-1)^{|\be|_1} \omega_{\bs - \be;\bk} v(\by_{\bs - \be;\bk}),
\end{equation*}
where	$\omega_{\bs;\bk}:= \prod_{j \in \supp (\bs)} \omega_{s_j;k_j}$ for 
$\bs \in \FF$ and $\bone \le \bk \le \bs$.

\begin{theorem}\label{thm:quadrature-X}
		Under the assumption of Theorem~\ref{thm:coll-approx-X},  
	 there exists a constant $C$ such that for each $n > 1$, we can construct a sequence of points  
		$(\by_{\bs - \be;\bk})_{(\bs,\be,\bk) \in G(\xi_n)}$ so that
		$|G(\xi_n)| \le n$ and
		\begin{equation*} \label{u-Q_Lambda u}
			\left\|\int_{\RRip}u(\by)\, \rd \blambda(\by )- Q_{\Lambda(\xi_n)}u\right\|_V
			\ \le \
			Cn^{-(1/p - 1)},
		\end{equation*}		
		and, if  additionally, $\phi \in V'$ is a bounded linear functional on $V$, 
		\begin{equation*} \label{u-Q_Lambda u-phi}
			\left|\int_{\RRip} \langle \phi, u (\by) \rangle\, \, \rd 
			\blambda(\by ) -  \left\langle \phi, Q_{\Lambda(\xi_n)} u \right\rangle \right|
			\ \le \
			C\norm{\phi}{V'} n^{-(1/p - 1)}.
		\end{equation*}	
\end{theorem}
\begin{proof}
	This theorem can be proven similarly to \cite[Corollary 4.1]{Dung21}  for functions in the Bochner space $L_2(\RRi,X;\bgamma)$.
	\hfill
\end{proof}	

\subsection{Extended least squares  approximation}

We analyze the convergence rate of the extended least-squares sampling algorithms for the solution $u(\by)$ to  equation \eqref{SPDE}--\eqref{lognormal}($\operatorname{\Gamma}$).
Let the assumptions of either Theorem~\ref{thm: summability-ell_p-rho} or  
Theorem~\ref{thm: summability-ell_p-b-L} hold for $0<p<2$. These theorems state that one can construct a set 
$\bsigma=(\sigma_\bs)_{\bs \in \FF}$, and a constant $M$  such that	
\begin{equation} \label{ell_2-summability}
	\left(\sum_{\bs\in\FF} (\sigma_\bs \|u_\bs\|_V)^2\right)^{1/2} \ \le M^{1/2} \ <\infty, \ \ \text{with} \ \
	\norm{\bsigma^{-1}}{\ell_q(\FF)} \le M^{1/q} < \infty,
\end{equation}	
where the set $\bsigma$ and constant $M$ depend on $a, \bb, p, \theta,  \lambda$ only. 
Provided this weighted $\ell_2$-summability,  the series \eqref{GPCexpansion} converges  unconditionally in 
$L_2(\RRip,V;\blambda)$ to $u$. This unconditional convergence can be established in the same way as the proof of \cite[Lemma 3.1]{Dung22}.  Hence, putting  $U:= \RRip$, $\sigma_j:=\sigma_{\bs_j}$, \ $\varphi_j:=L_{\bs_j}$ and $u_j:=u_{\bs_j}$,  by \eqref{ell_2-summability} we can reorder the countable set $\FF$ as $\FF = (\bs_j)_{j \in \NN}$ so that the sequence 
$(\sigma_{\bs_j})_{j \in \NN}$ is non-decreasing, and the weak solution $u$ is represented by the series
\begin{equation*} 
	u = \sum_{j \in \NN}  u_j\, \varphi_j,
\end{equation*}
with
\begin{equation} \label{w-summability}
	\left(\sum_{j \in \NN} (\sigma_j \|u_j\|_V)^2\right)^{1/2} \ \le M^{1/2}, \ \ \text{with} \ \
	\norm{\brac{\sigma_j^{-1}}_{j \in \NN}}{\ell_q(\NN)} \le M^{1/q}.
\end{equation}
Notice that $u(\by)$ is well-defined for every $\by\in U_0$, where $U_0$ is the set defined as 
 in \eqref{U_0} for $U=\RRip$. By Lemma \ref{lemma:EE}, the set $U_0$ has full measure, i.e., $\blambda(U_0)=1$.

Let us construct an extension of a least squares approximation in the space $L_2(\RRip,\CC;\blambda)$ to a space $L_2(\RRip,X;\blambda)$ (for detail in the general setting, see \cite{BD2024}).
For $n,m\in\NN$ with $n\ge m$, let $\by_1, \dots, \by_{n}\in U_0$ be points, $\omega_1, \dots, \omega_{n}\ge 0$ be weights, and $V_m = \operatorname{span}\{\varphi_{j}\}_{j=1}^{m}$ the subspace spanned by the functions $\varphi_{j}$, $j=1,...,m$.
The weighted least squares approximation $S_{n}^{\CC} f = S_{n}^{\CC}(\by_1, \dots, \by_{n}, \omega_1, \dots, \omega_{n}, V_m) f$ of a function $f\colon \RRip\to\CC$ is given by
\begin{equation} \label{least-squares-sampling1}
	S_{n}^{\CC} f
	= \underset{{g\in V_m}}{\operatorname{arg\,min}} \sum_{i=1}^{n} \omega_i |f(\by_i) - g(\by_i)|^2 .
\end{equation}
For every $n \in \NN$, let
\begin{equation*} \label{S_nCC}
	S_{n}^\CC f : = \sum_{i=1}^{n} f(\by_i) h_i, 
\end{equation*}	
be the least squares sampling algorithm constructed by 
\eqref{least-squares-sampling1}  for these sample points and weights, where 
$h_1,...,h_{n} \in L_2(\RRip,\CC;\blambda)$.
We define the extension of this least squares algorithm to the Bochner space $L_2(\RRip,X;\blambda)$
by replacing $f\in L_2(\RRip,\CC;\blambda)$ with $v\in L_2(\RRip,X;\blambda)$:
\begin{equation*} \label{S_nX}
	S_{n}^X v
	:= 	
	S_n^X(\by_1, \dots, \by_n, \omega_1, \dots, \omega_n, V_m) v
	:= \sum_{i=1}^{n} v(\by_i) h_i.
\end{equation*}
As the least squares approximation is a linear operator, worst-case error bounds carry over from the usual Lebesgue space $L_2(\RRip,\CC;\blambda)$ to the Bochner space $L_2(\RRip,X;\blambda)$.
We define $	\tilde{S}_n^X$ by
	\begin{equation*}\label{tilde{S}_n^X:=}
	\tilde{S}_n^X v
	:= 	
S_n^X(\by_1, \dots, \by_n, \omega_1, \dots, \omega_n, V_m) v
		:= \sum_{i=1}^{n} v(\by_i) h_i
	, \ \  \text{for} \ \
	m:=\Big\lceil \frac{n}{43200}
	\Big\rceil. 
\end{equation*}
\begin{theorem} \label{thm: LS-sampling}
	Let the assumptions of either Theorem~\ref{thm: summability-ell_p-rho} or  
	Theorem~\ref{thm: summability-ell_p-b-L} hold for $0<p<2$. 	Then there exist a constant $C$ such that for all $n \in \NN$,
	\begin{itemize}
		\item[{\rm (i)}] 
	\begin{equation*}\label{u - S_{n}^V u-L}
		\norm{u - \tilde{S}_n^V u}{L_2(\RRip,V;\blambda)}
		\ \le CM n^{-(1/p -1/2)};
	\end{equation*}	
	\item[{\rm (ii)}] 
	moreover, for the quadrature 
	\begin{equation*} \label{S_nV-L}
		\tilde{Q}_n^V u
		= \sum_{i=1}^n w_i u(\by_i), \quad   w_i:= \int_{\RRip}  \tilde{h}_i(\by) \rd \blambda(\by),
	\end{equation*}
	it holds
	\begin{equation} \label{u-Q-LS}
		\left\|\int_{\RRip}u(\by)\, \rd \blambda(\by ) - \tilde{Q}_n^V u\right\|_V
		\ \le \
		CM n^{-(1/p -1/2)},
	\end{equation}
	and, if  additionally, $\phi \in V'$ be a bounded linear functional on $V$, 
	\begin{equation} \label{u-Q_Lambda u-phi-L}
		\left|\int_{\RRip} \langle \phi, u (\by) \rangle\, \, \rd 
		\blambda(\by ) - \left\langle \phi,  \tilde{Q}_n^V  u \right\rangle \right|
		\ \le \
		C\norm{\phi}{V'} n^{-(1/p -1/2)}.
	\end{equation}
		\end{itemize}			
	\end{theorem}	
	
	\begin{proof}
		Claim (i) is directly derived from \cite[Corollary 2.1]{BD2024} and the weighted 
		$\ell_2$-summability~\eqref{w-summability}. Let us prove claim (ii). From the equality 
\begin{equation*} 
\int_{\RRip}u(\by)\, \rd \blambda(\by ) - \tilde{Q}_n^V u
\ = \
\int_{\RRip}\brac{u(\by)- \tilde{S}_n^V u(\by)}\rd \blambda(\by ), 
\end{equation*}	
and the claim (i)	it follows that
\begin{equation*} 
\begin{aligned}
	 	\left\|\int_{\RRip}u(\by)\, \rd \blambda(\by ) - \tilde{Q}_n^V u\right\|_V
	 \ &\le \	
	 \norm{u - \tilde{S}_n^V u}{L_1(\RRip,V;\blambda)}
	 \\
	  & \le \ 
	 \norm{u - \tilde{S}_n^V u}{L_2(\RRip,V;\blambda)}
	 \ \le \ 
	 CM n^{-(1/p -1/2)},
\end{aligned}
\end{equation*}
which proves \eqref{u-Q-LS}. In a similar way, we can establish \eqref{u-Q_Lambda u-phi-L} by using the equality
\begin{equation*} 
\int_{\RRip} \langle \phi, u (\by) \rangle\, \, \rd 
\blambda(\by ) - \left\langle \phi,  \tilde{Q}_n^V  u \right\rangle 
	\ = \
	\int_{\RRip} \left\langle \phi, u(\by)- \tilde{S}_n^V u(\by) \right\rangle \, \, \rd \blambda(\by ).
\end{equation*}	
		\hfill
	\end{proof}
	
		\begin{remark} \label{remark3.1}
		{\rm
	The convergence rate of the approximation of $u$ by least squares sampling algorithms  (Theorem~\ref{thm: LS-sampling}) is markedly superior to that of the approximation of $u$ by sparse-grid polynomial interpolations (Theorem~\ref{thm:coll-approx-X}) by a factor of $n^{-1/2}$. The former matches the optimal  rate of best $n$-term polynomial approximation of $u$,  which can be shown via  Stechkin's lemma (see, e.g., \cite[Section 3]{CoDe15a}).
	}	
\end{remark}

\section{Parametric PDEs with Gaussian log-normal inputs}
\label{Parametric PDEs with Gaussian log-normal inputs}

In this section, we present a counterpart of our  method for  parametric elliptic PDEs with Gaussian log-normal inputs.  Throughout the section, we denote by $u$ or   $u(\by)$  the weak solution  to equation \eqref{SPDE}--\eqref{lognormal}(G).

Let $\gamma$ be the standard Gaussian probability measure on $\RR$ 
with the density 
\begin{equation*} \label{g}
	g(y):=\frac 1 {\sqrt{2\pi}} e^{-y^2/2} .
\end{equation*}
Let $(H_k)_{k\in \NN_0}$ be the sequence of Hermite polynomials on $\RR$ 
normalized with respect to the measure $\gamma$, i.e., 
$$
\int_{\RR} |H_k(y)|^2 \rd \gamma(y) =\int_{\RR} |H_k(y)|^2 g(y) \rd y =1, \ \
k\in \NN_0.
$$

The the Bochner space  $L_2(\RRi,X;\bgamma)$ and the associated infinite tensor product standard Gaussian measure $\bgamma(\by)$ on $\RRi$ can be defined in the same fashion as the Bochner space  $L_2(\RRip,X;\blambda)$ and the measure $\blambda(\by)$.

A function $v \in L_2(\RRi,X;\bgamma)$ can be represented by the Gaussian GPC expansion
\begin{equation*} \label{GPCexpansion-G}
	v(\by)=\sum_{\bs\in\FF} v_\bs \,H_\bs(\by), \quad v_\bs \in X,
\end{equation*}
with convergence in $L_2(\RRi,X;\bgamma)$, where
\begin{equation*}
	H_\bs(\by)=\bigotimes_{j \in \NN}H_{s_j}(y_j),\quad 
	v_\bs:=\int_{\RRi} v(\by)\,H_\bs(\by)\, \rd\bgamma (\by), \quad 
	\bs \in \FF.
\end{equation*}

It is known that, the Hermite polynomials $\big(H_k\big)_{k\in \NN_0}$ are nontrivial solution to
Hermite's (linear second-order) differential equation
\begin{equation*} \label{Dd-Hermite} 
	\Dd u = k u, \quad \text{where} \quad \Dd   := -  \frac{\rd^2  }{\rd y^2} + y \frac{\rd  }{\rd y}
\end{equation*}
(see, e.g., \cite[5.5.2]{Szego1939}).
This means that 
\begin{equation*} \label{bb-H} 
	\Dd   H_k   = H_k,  \ \ k \in \NN_0. 
\end{equation*}
Analogously to the Laguerre polynomials, it holds the equality 
\begin{equation*} \label{Dd2-H} 
	\Dd u = - e^{y^2/2} \frac{\rd }{\rd  y} \brac{e^{-y^2/2}  \frac{\rd u}{\rd y}}.
\end{equation*}

For a given finite set $J \subset \NN$ and $r \in\NN$, we define the product differential operator 
\begin{equation*} \label{Dd^r_J-gamma} 
	\Dd^r_J  := \brac{\Dd_J}^r, \quad 
	\Dd_J  := (-1)^{|J|} \prod_{j \in J} e^{y_j^2/2}
	\frac{\rd}{\rd y_j} \brac{e^{-y_j^2/2}  \frac{\rd}{\rd y_j}}.
\end{equation*}
Notice that there exist polynomials $q_j(y)$, $j = 1,...,2r$, of degree at most $r$  such that 
	\begin{equation*} \label{Dd^r-gamma} 
		\Dd^r 
		= \sum_{j=1}^{2r}	q_j(y) \frac{\rd^j }{\rd y^j},
	\end{equation*}
	and the coefficients of  $q_j(y)$ depend on $r$ only.	Hence, there exists a constant $C_{r}$ such that
	\begin{equation} \label{q_j} 
	|q_j(y)|  \le C_{r}(1 + |y|)^r, \ \ j = 1,\ldots,2r.
	\end{equation}		

Hence, in a way similar to the proof of Lemma~\ref{lemma:Dd^r_J v =Laguerre} we can prove
\begin{lemma} \label{lemma:Dd^r_J v =Hermite}
		If $v \in L_2(\RRi,V;\bgamma)$ and $\Dd_J^r v \in L_2(\RRi,V;\bgamma)$, then we have that
	\begin{equation*} \label{identity1-H} 
 \Dd^r_J v = \sum_{\bs \in \FF} \nu_{J,\bs}^r v_\bs H_\bs, 
	\end{equation*}
 where $\nu_{J,\bs}$ is defined in \eqref{nu_J}. 
\end{lemma}

We recall the following result  proven in \cite[Theorem 2.2]{BCDM17}.

\begin{lemma} \label{lemma:exp(k|b|)<}
Assume that 
	\begin{equation*}\label{AssumA}
		\text{$\exists$ $(\rho_j)_{j\in \NN}$, $\rho_j >0$: \
			$\sum_{j \in \NN} \rho_j|\psi_j|$  converges  in  $L_\infty(D)$  and	\ \
			$\sum_{j \in \NN} \exp\big(-\rho_j^2\big) < \infty$.} 
	\end{equation*}	
Then for any $k\in \NN$
$$
\int_{\RRi} \exp\big(k\|b(\by)\|_{L_\infty(D)}\big)\rd\bgamma(\by) <\infty.
$$
\end{lemma}
	We put
\begin{equation*}
	\tilde{A}_r(J):= \brac{\int_{\RRi}\prod_{j \in J} (1 + y_j)^{2r} 
		\exp\Big(2 \|b(\by)\|_{L_\infty(D)}\Big) \rd \bgamma(\by)}^{1/2}.
\end{equation*}
Under the assumption of Lemma \ref{lemma:exp(k|b|)<},  by H\"older's inequality we deduce that
\begin{equation} \label{A_r(J)^4}
\tilde{A}_r(J)^4 
\ \leq \
\int_{\RRi}\exp\big(4\|b(\by)\|_{L_\infty(D)}\big)\rd\bgamma(\by) \brac{\int_{\RR} (1+y)^{4r}\rd\gamma(y)}^{|J|} .
\end{equation}

The following result can be established analogously, with appropriate modifications to the proof of Theorem \ref{thm: summability-ell_p-rho} by using, in particular, inequality \eqref{A_r(J)^4}.

\begin{theorem} \label{thm: summability-ell_p-rho-H}
	Let  $0< p < \infty$, $\brho=(\rho_j)_{j\in \NN}$ be a positive sequence satisfying  condition~\eqref{<kappa}, and $\brho^{-1}\in \ell_p(\NN)$. 	
	Then $\brac{\norm{u_\bs}{V}}_{\bs\in \FF}\in \ell_p(\FF)$.	
	
	Moreover, if in addition, $p < 2$, for any fixed $\theta, \lambda \ge 0$,  we can construct a set 
	$\bsigma=(\sigma_\bs)_{\bs \in \FF}$  with positive $\sigma_\bs$, and a constant $M$ such that	
	\begin{equation*} \label{ell_2-summability-rho-H}
		\left(\sum_{\bs\in\FF} (\sigma_\bs \|u_\bs\|_V)^2\right)^{1/2} \ \le M^{1/2} \ <\infty, \ \ \text{with} \ \
		\norm{\bp(\theta,\lambda)\bsigma^{-1}}{\ell_q(\FF)} \le M^{1/q} < \infty,
	\end{equation*}	
	where $q := 2p/(2-p)$.  The set $\bsigma$ and constant $M$ depend on $a,  p, \theta, \lambda, \kappa$ only. 
\end{theorem}

We have the following estimate for $\tilde{A}_r(J)$  in the case of arbitrary supports of $\psi_j$.

	\begin{lemma} \label{lemma:A_r(J)-2}
	Let  $\bb = \brac{b_j}_{j \in \NN}$ be defined by 
	$b_j:= \norm{\psi_j}{L_\infty(D)}$. Assume 
	$\bb\in \ell_1(\NN)$ and
	\begin{equation*}
		\norm{\bb}{\ell_\infty(\NN)} = b_0 < \infty.
	\end{equation*}
	Then 
	\begin{equation*}
		\tilde{A}_r(J)\le 
		K_{r,\bb}^{|J|} \exp\bigg(\|\bb\|_{\ell_2(\NN)}^2 +  \frac{\sqrt{2}}{\pi}\|\bb\|_{\ell_1(\NN)}\bigg),
	\end{equation*}
	where
	\begin{equation} \label{K_{r,bb}}
		K_{r,\bb}:= \brac{\int_{\RR}  \frac{(1+y)^{2r} }{\sqrt{2\pi}} \exp\bigg(2b_0|y| -\frac{|y|^2}{2}\bigg) \rd y}^{1/2} <\infty.
	\end{equation}
\end{lemma}

\begin{proof}
Indeed,	by using the inequality
		\begin{equation*}
		\int_\RR \exp(b|y|)\rd \gamma(y) \le \exp\bigg(\frac{b^2}{2}+\frac{\sqrt{2}b}{\pi}\bigg),
	\end{equation*}
	(see, e.g.,  \cite[(38)]{BCDM17}),  we derive
	for any finite set $J\subset \NN$,
		\begin{equation*}
		\begin{aligned}
			\tilde{A}_r(J)^2
			&\leq  \int_{\RRi}\prod_{j \in J} (1 + y_j)^{2r} 
			\exp\Bigg(2 \sum_{j\in \NN}b_jy_j\Bigg) \prod_{j\in \NN}  \rd \gamma( y_j)
			\\
			&\leq \prod_{j\in J} \int_{\RR}  \frac{(1+y_j)^{2r} }{\sqrt{2\pi}} \exp\bigg(2b_0|y_j| -\frac{|y_j|^2}{2}\bigg) \rd y_j \prod_{j\not\in J} \int_{\RR} \exp\big(2b_j|y_j|  \big) \rd \gamma(y_j).
				\\
			&\leq K_{r,\bb}^{2|J|} \prod_{j\not\in J} 
			\exp\bigg(2b_j^2+\frac{2\sqrt{2}b_j}{\pi}\bigg)
			\leq K_{r,\bb}^{2|J|} \exp\bigg(2\|\bb\|_{\ell_2(\NN)}^2 +  \frac{2\sqrt{2}}{\pi}\|\bb\|_{\ell_1(\NN)}\bigg).
		\end{aligned}
	\end{equation*}	
	\hfill	
\end{proof}

In a manner analogous  to  the proof of 
Theorem~\ref{thm: summability-ell_p-b-L},  with certain modifications we can obtain the following result.

\begin{theorem} \label{thm: summability-ell_p-b-H}
		Let $0< p \le 1$ and  $\bb = \brac{b_j}_{j \in \NN}$ be defined by	$b_j:= \norm{\psi_j}{L_\infty(D)}$. 		Let $r=r_{p,0}\in \NN$ be a fixed number satisfying \eqref{C_{p,theta}}.
Assume  that $\bb\in \ell_p(\NN)$ and
		\begin{equation*} 
			\norm{\bb}{\ell_1(\NN)}	< K^{-1},
		\end{equation*}
		where 
	\begin{equation} \label{K-for-gamma}
			K:= e (2r)!C_{r}C_{p,0}K_{r,\bb}, 
		\end{equation}
		and the constants $C_{r}$,   $C_{p,0}$, $K_{r,\bb}$  are defined as in \eqref{q_j}, \eqref{C_{p,theta}}, \eqref{K_{r,bb}} respectively.
	Then $\brac{\norm{u_\bs}{V}}_{\bs \in \FF}\in \ell_p(\FF)$. 
	
	Moreover, 
for any fixed $\theta, \lambda \ge 0$,  if  $q := 2p/(2-p)$,  $\theta'=2\theta /(2-p)$, $r= r_{p,\theta'}$ in \eqref{C_{p,theta}} and $K$ in \eqref{K-for-gamma} is replaced by 
		\begin{equation*} \label{K-gamma}
			K:= e(2r)!C_{r}C_{p,\theta'}C_{\theta',\lambda}K_{r,\bb}
		\end{equation*}  
with $C_{p,\theta'}$ being defined as in \eqref{ineq-theta,lambda}, then	we can construct a set 
	$\bsigma=(\sigma_\bs)_{\bs \in \FF}$, and a constant $M$  such that	
	\begin{equation*} \label{ell_2-summability-H}
		\left(\sum_{\bs\in\FF} (\sigma_\bs \|u_\bs\|_V)^2\right)^{1/2} \ \le M^{1/2} \ <\infty, \ \ \text{with} \ \
		\norm{\bp(\theta,\lambda)\bsigma^{-1}}{\ell_q(\FF)} \le M^{1/q} < \infty,
	\end{equation*}	
	where  the set $\bsigma$ and constant $M$ depend on $\bb, p, \theta, \lambda$ only. 
\end{theorem}

	\begin{remark} \label{remark4.1}
		{\rm Theorem~\ref{thm: summability-ell_p-rho-H} presents a sufficient condition for the $\ell_p$-summability $\brac{\norm{u_\bs}{V}}_{\bs\in \FF}\in \ell_p(\FF)$ for $0 < p < \infty$. Sufficient conditions for this $\ell_p$-summability considered 
				in \cite[Proposition 4.4]{HoSc14} for $0<p\le1$, and 
				in \cite[Theorem 1.2]{BCDM17} for $0<p<2$.
			For $0 <p<2$, the $\ell_p$-summability result of  \cite[Theorem 1.2]{BCDM17}  under the condition
			$\brho^{-1} \in \ell_q(\NN)$ with $q := 2p/(2-p)$, is better than the $\ell_p$-summability result of Theorem~\ref{thm: summability-ell_p-rho-H}.
			
			Consider the component function $\psi_j$ with arbitrary supports  in the parametric equation with log-normal random inputs  
			\eqref{SPDE}--\eqref{lognormal}(G), including the case of globally supported functions such as the Fourier series.
In this context, the sufficient condition  
		$\big(\norm{\psi_j}{L_\infty(D)}\big)_{j \in \NN}\in \ell_p(\NN)$ in 
		Theorem~\ref{thm: summability-ell_p-b-H} for the $\ell_p$-summability $\big(\norm{u_\bs}{V}\big)_{\bs \in \FF}\in \ell_p(\FF)$, significantly improves the sufficient conditions $\big(\norm{j\psi_j}{L_\infty(D)}\big)_{j \in \NN}\in \ell_p(\NN)$ in \cite{HoSc14},
		and $\big(\norm{j^\alpha\psi_j}{L_\infty(D)}\big)_{j \in \NN}\in \ell_p(\NN)$ with some 
		$\alpha > 1/2$, in \cite[Corollary 6.3]{BCDM17} and \cite[Section 3.6.2]{DNSZ2023}.
	}
	\end{remark}

 \medskip
\noindent
{\bf Acknowledgments:} 
The work  of Dinh D\~ung and Van Kien Nguyen is funded by the Vietnam National Foundation for Science and Technology Development (NAFOSTED) under  the Vietnamese--Swiss Joint Research Project, Grant No. IZVSZ2$_{ - }$229568. Viet Ha Hoang is supported by the Tier 1 grant RG24/23 and the Tier 2 grant MOE-T2EP20123-0008 (proposal T2EP20123-0047) awarded by the Singapore Ministry of Education.
A part of this work was done when the authors were working at the Vietnam Institute for Advanced Study in Mathematics (VIASM). They would like to thank the VIASM for providing a fruitful research environment and working condition.

\bibliographystyle{abbrv}
\bibliography{AllBib}

\end{document}